\documentclass[12pt, twoside, leqno]{amsart}

\usepackage{amssymb,amsthm,amsmath,amscd}     %%%% optional
\usepackage{latexsym}                         %%%% optional
\usepackage{enumerate}
\usepackage{url}

\usepackage{verbatim} %\usepackage{showkeys} % shows labels and references
\usepackage{color} % for colored background, colored text

\newtheorem{thm}{Theorem}[section]
\newtheorem{cor}[thm]{Corollary}
\newtheorem{lem}[thm]{Lemma}

\theoremstyle{definition}
\newtheorem{defn}[thm]{Definition}
\newtheorem{rem}[thm]{Remark}

\newcommand{\C}{\ensuremath{\mathbb{C}}}
\newcommand{\E}{\ensuremath{\mathrm{e}}}
\newcommand{\N}{\ensuremath{\mathbb{N}}}

\newcommand{\R}{\ensuremath{\mathbb{R}}}

\newcommand{\Z}{\ensuremath{\mathbb{Z}}}

\newcommand{\bfb}{\ensuremath{\mathbf b}}
\newcommand{\bfc}{\ensuremath{\mathbf c}}

\newcommand{\bfg}{\ensuremath{\mathbf g}}
\newcommand{\bfk}{\ensuremath{\mathbf k}}

\newcommand{\bfx}{\ensuremath{\mathbf x}}
\newcommand{\bfy}{\ensuremath{\mathbf y}}

\newcommand{\bfzero}{\ensuremath{\mathbf 0}}

\newcommand{\myone}{\mathbf{1}}

\newcommand{\sprod}{\prod_{i=1}^s}

\newcommand{\Is}{[0,1)^s}

\newcommand{\etk}{Erd\"os-Tur\'an-Koksma}

\begin{document}
% --------------

\title
[Hybrid inequality]{A hybrid inequality of Erd\"os-Tur\'an-Koksma for digital sequences}

\author{PETER HELLEKALEK}
\address{Peter Hellekalek,
Dept. of Mathematics,
University of Salzburg,
Hellbrunnerstrasse 34,
5020 Salzburg\\
AUSTRIA}
\email{peter.hellekalek@sbg.ac.at}

\date{\today}

\begin{abstract}
For bases $\mathbf{b}=(b_1, \ldots, b_s)$ of $s$ not necessarily
distinct integers $b_i\ge 2$,
we prove a version of the inequality of \etk \ for the hybrid function system
composed of the Walsh functions in base $\bfb^{(1)}=(b_1, \ldots, b_{s_1})$
and, as  second component,
the $\bfb^{(2)}$-adic functions, $\bfb^{(2)}=(b_{s_1+1}, \ldots, b_s)$,
with $s=s_1+s_2$, $s_1$ and $s_2$ not both equal to 0.
Further,
we point out why this choice of a hybrid function system covers all possible cases of sequences
that employ addition of digit vectors as their main construction principle.
\end{abstract}

\maketitle

%% Classification and key words:

\renewcommand{\thefootnote}{}



\renewcommand{\thefootnote}{\arabic{footnote}}
\setcounter{footnote}{0}

\section{Introduction}
% --------------------
\label{s:intro}
In this paper we exhibit a version of the inequality of \etk\
tailored to hybrid digital sequences.
Our result may be seen as a complement to a recent theorem of Niederreiter\cite[Theorem 1]{Nie10a}.

{\em Hybrid sequences} are sequences of points in the multidimensional unit cube $[0,1)^s$
where certain coordinates of the points stem from one lower-dimensional sequence
and  the remaining coordinates from a second lower-dimensional sequence.
%This idea of ``mixing'' two sequences to obtain a new sequence in
%higher dimensions may be extended easily to more than two components.

The analysis of the uniformity of hybrid sequences requires new tools.
The classical results of the theory of uniform distribution of sequences have to be
adapted to the fact that, in a hybrid sequence,
several types of arithmetic may be involved.
For example, the existing versions of the inequality of \etk\ are unsuited to
accommodate for hybrid sequences because only one type of exponential sum
is involved in each version
(see \cite[Ch. 3.2]{Nie92a}, \cite{Hel93a,Hel09a} for details).
The first hybrid version of the inequality of \etk\ was established in Niederreiter\cite{Nie10a}
and has found numerous applications (see \cite{Nie09a,Nie10a,Nie10b,Nie11a,Nie11b}).
For hybrid versions of the Weyl criterion and of diaphony we refer the reader to \cite{Hel10c,Hel11a}.

{\em Digital sequences} in $[0,1)^s$ are constructed from the representation of real numbers
or integers in  given integer  bases. %$b_1, b_2, \ldots, b_s$.
There exist two types of such sequences,
sequences  using addition without carry of digit vectors
and sequences that employ addition with carry of such vectors in their construction method.
Important examples of the first type are digital  nets,
see Niederreiter  \cite[Ch. 4]{Nie92a} and Dick and Pillichshammer~\cite{Dick10a}.
For the second type,
an important example are the Halton sequences,
which are generated by addition with carry.
The underlying group is the compact abelian group of $\bfb$-adic integers
(see \cite[p.29]{Nie92a} and \cite[Cor. 2.18]{Hel11a}).

Our results  come within the framework of a series of studies initiated in
\cite{Hel09a,Hel10a,Hel10c,Hel11a} on the {\em $b$-adic method} in the theory
of uniform distribution of sequences.
In this method,
we employ structural properties of the compact group of $b$-adic integers
as well as $b$-adic arithmetic to derive tools for the analysis of
sequences  in the $s$-dimensional unit cube $[0,1)^s$.

The reason why we have selected  Walsh functions and  $b$-adic functions
in our version of the \etk \ inequality is based on the fact that there exist only two basic types
of addition for digit vectors,
addition with and without carry (see \cite{Hel12c} and Section \ref{s:addition} below).
As a consequence,
when it comes to analyze sequences based on arithmetic with digit vectors,
these two function systems suffice.

In addition to the hybrid \etk\ inequality in Theorem \ref{thm:hybridETK},
we extend techniques of  \cite{Hel09a,Hel10a} from the case of prime bases to
arbitrary bases $\bfb=(b_1, \ldots, b_s)$ consisting of
not necessarily distinct integers $b_i\ge 2$ (see Section~\ref{s:badic}).

\section{Preliminaries}
\label{s:badic}
%$\Gamma_\mathbf{b}^{(s)}$}
% ---------------------------------------------------------------------------
Throughout this paper,
$b$ denotes a positive integer, $b\ge 2$,
and $\mathbf{b}=(b_1, \ldots, b_s)$ stands for a vector of not necessarily distinct
integers $b_i\ge 2$, $1\le i\le s$.
$\mathbb{N}$ represents the positive integers,
and we put $\mathbb{N}_0=\mathbb{N}\cup \{0\}.$

The underlying space is the $s$-dimensional torus
$\mathbb{R}^s/\mathbb{Z}^s$,
which will be identified with the half-open interval $[0,1)^s$.
Haar measure on the $s$-torus $[0,1)^s$ will be denoted by $\lambda_s$.
We put $e(y)=\E^{2\pi \mathrm{i} y}$ for $y\in\mathbb{R}$,
where $\mathrm{i}$ is the imaginary unit.

We will use the standard convention that empty sums have the value 0 and empty products
value 1.

For a nonnegative integer $k$, let
$k =\sum_{j\ge0} k_j\, b^j, k_j \in \{0,1,\ldots ,b-1\},$
be the unique $b$-adic representation of $k$ in base $b$. With the
exception of at most finitely many indices $j$, the digits $k_j$ are
equal to 0.

Every real number $x\in [0,1)$ has a $b$-adic representation
of the form
$x = \sum_{j\ge 0} x_j \,b^{-j-1},$
with digits
$x_j\in \{0,1, \ldots, b-1\}.$
If $x$ is a {\em $b$-adic rational},
which means that $x=ab^{-g}$, $a$ and $g$ integers, $0\le a< b^g$,
$g\in \mathbb{N}$,
and if $x\ne 0$,
then there exist two such representations.
%one of them with the property that $x_j=0$ for all $j$ sufficiently large,
%the other one with $x_j=p-1$ for all $j$ sufficiently large.

The $b$-adic representation of $x$ is uniquely determined
under the condition that $x_j \ne b-1$ for infinitely many $j$.
In the following,
we will call this particular representation the {\em regular} ($b$-adic) representation
of $x$.

Let $\mathbb{Z}_b$ denote the compact group of the $b$-adic integers. We refer
the reader to Hewitt and Ross \cite{Hewitt79a} and Mahler \cite{Mahler81a} for details.
An element $z$ of $\mathbb{Z}_b$ will be written in the form
$z=\sum_{j\ge 0} z_j\, b^j,$
with digits $z_j\in \{0,1,\ldots, b-1  \}$.
The set $\mathbb{Z}$ of integers  is embedded in $\mathbb{Z}_b$.
If $z\in \mathbb{N}_0$,
then at most finitely many digits $z_j$ are different from 0.
If $z\in \mathbb{Z}$, $z< 0$,
then at most finitely many digits $z_j$ are different from $b-1$.
In particular,
$-1 = \sum_{j\ge 0} (b-1)\, b^j.$

We recall the following concepts from \cite{Hel10a,Hel11a}.
\begin{defn}
The map $\varphi_b: \mathbb{Z}_b \rightarrow [0,1)$,
given by
$\varphi_b(\sum_{j\ge 0} z_j\, b^j) = \sum_{j\ge 0} z_j\, b^{-j-1}\pmod{1}$,
will be called the {\em $b$-adic Monna map}.
\end{defn}

The restriction of $\varphi_b$ to $\mathbb{N}_0$ is often called the {\em radical-inverse function}
in base $b$.
The Monna map is surjective, but not injective.
It may be inverted in the following sense.
\begin{defn}
We define the {\em pseudoinverse} $\varphi^+_b$ of the $b$-adic Monna map $\varphi_b$  by
\[
\varphi^+_b: [0,1) \rightarrow \mathbb{Z}_b, \quad
\varphi^+_b(\sum_{j\ge 0} x_j\, b^{-j-1}) = \sum_{j\ge 0} x_j\, b^{j}\;,
\]
where $\sum_{j\ge 0} x_j\, b^{-j-1}$ stands for the regular $b$-adic
representation of the element $x\in [0,1)$.
\end{defn}

The image of  $[0,1)$ under $\varphi^+_b$ is the set  $\mathbb{Z}_b\setminus (-\mathbb{N})$.
Furthermore, $\varphi_b\circ \varphi^+_b$ is the identity map on $[0,1)$,
and $\varphi^+_b\circ \varphi_b$ the identity on $\mathbb{N}_0 \subset \mathbb{Z}_b$.
In general,
$z\neq \varphi^+_b(\varphi_b(z))$,  for $z\in \mathbb{Z}_b$.
For example,
if $z=-1$,
then $\varphi^+_b(\varphi_b(-1))=\varphi^+_b(0)=0\neq -1$.

It has been shown in \cite{Hel11a} that the dual group
 $\hat{\mathbb{Z}}_b$
can be written in the form
$\hat{\mathbb{Z}}_b= \{\chi_k: k\in \mathbb{N}_0 \},$
where
$\chi_k: \mathbb{Z}_b\rightarrow \{ c\in \mathbb{C}: |c|=1 \}$,
$\chi_k( \sum_{j\ge 0} z_j b^j ) =     e(\varphi_b(k)(z_0+z_1b+\cdots))$.
We note that $\chi_k$ depends only on a finite number
of digits of $z$ and, hence, this function is well defined.

As in \cite{Hel10a},
we employ the function $\varphi_b^+$ to lift the characters $\chi_k$ to the torus.

\begin{defn}
For $k\in \N_0$,
let
$\gamma_k: [0,1) \rightarrow \{ c\in \mathbb{C}: |c|=1 \}$,
$\gamma_k(x)=     \chi_k(\varphi^+_b(x))$,
denote the $k$th $b$-adic function.
We put $\Gamma_b=\{\gamma_k:  k\in \mathbb{N}_0  \}$
and call it the $b$-adic function system on $[0,1)$.
\end{defn}

There is an obvious generalization of the preceding notions to the
higher-dimensional case.
Let $\mathbf{b}=(b_1, \ldots, b_s)$ be a vector of
not necessarily distinct integers $b_i\ge 2$,
let ${\mathbf{x}}= (x_1, \ldots, x_s)\in [0,1)^s$,
let $\mathbf{z}=(z_1, \ldots, z_s)$ denote an element of the compact product group
$\mathbb{Z}_\mathbf{b}= \mathbb{Z}_{b_1}\times \cdots \times \mathbb{Z}_{b_s}$
of $\mathbf{b}$-adic integers,
and let ${\mathbf{k}}= (k_1, \ldots, k_s)\in \mathbb{N}_0^s$.
We define
$\varphi_{\mathbf{b}}(\mathbf{z}) = (\varphi_{b_1}(z_1), \ldots, \varphi_{b_s}(z_s))$,
and
$\varphi^+_{\mathbf{b}}(\mathbf{x}) = (\varphi^+_{b_1}(x_1), \ldots, \varphi^+_{b_s}(x_s))$.

Let
$\chi_{\mathbf{k}}(\mathbf{z}) = \prod_{i=1}^s \chi_{k_i}(z_i)$,
where $\chi_{k_i} \in \hat{\mathbb{Z}}_{b_i}$,
and define
$\gamma_{\mathbf{k}}({\mathbf{x}}) = \prod_{i=1}^s \gamma_{k_i}(x_i)$,
where $\gamma_{k_i} \in \Gamma_{b_i}$, $1\le i\le s$.
Then
$\gamma_{\mathbf{k}} = \chi_{\mathbf{k}}\circ \varphi^+_{\mathbf{b}}$.
Let $\Gamma_{\mathbf{b}} = \{ \gamma_{\mathbf{k}}:
\mathbf{k} \in \mathbb{N}_0^s \}$
denote the {\em $\mathbf{b}$-adic function system} in dimension $s$.

The dual group $\hat{\Z}_\bfb$ is an orthonormal basis of
the Hilbert space $L^2(\Z_\bfb)$.
It would follow from some measure-theoretic arguments that $\Gamma_\bfb$ is
an orthonormal basis of $L^2([0,1)^s)$.
A rather elementary proof of this result is given in \cite[Theorem 2.12]{Hel11a}.

\begin{defn}
For $k\in \mathbb{N}_0$,
$k= \sum_{j\ge 0} k_j b^j$,
and $x\in [0,1)$, with regular $b$-adic representation
$x= \sum_{j\ge 0} x_j b^{-j-1}$,
the {\em $k$th Walsh function in base $b$} is defined by
$w_k(x) = e( (\sum_{j\ge 0} k_j x_j)/b)$.
For ${\mathbf{k}}\in \mathbb{N}_0^s$, ${\mathbf{k}}=(k_1, \ldots, k_s)$,
and ${\mathbf{x}}\in [0,1)^s$,
${\mathbf{x}}=(x_1, \ldots, x_s)$,
we define the {\em ${\mathbf{k}}$th Walsh function $w_{\mathbf{k}}$
in base $\mathbf{b}=(b_1, \ldots, b_s)$} on  $[0,1)^s$ as the following product:
$w_{\mathbf{k}}(\mathbf{x}) = \prod_{i=1}^s
	w_{k_i}(x_i)$,
where $w_{k_i}$ denotes the $k_i$th Walsh function in base $b_i$,
$1\le i\le s$.
The Walsh function system in base $\mathbf{b}$,
in dimension $s$, is denoted by
$\mathcal{W}_{\mathbf{b}} = \{w_{\mathbf{k}}:   \mathbf{k}\in \mathbb{N}_0^s\}$.
\end{defn}

We refer the reader to \cite{Dick10a,Hel93a,Hel98a} for
elementary properties of the Walsh functions and to
\cite{Sch90a} for the background in harmonic analysis.

The following notion is a special case of the concepts discussed in \cite{Hel10c}.
For given dimensions $s_1$ and $s_2$,
with $s_1, s_2 \in \mathbb{N}_0$, not both equal to 0,
put $s=s_1+s_2$ and write a point $\bfy\in \R^s$ in the form $\bfy=(\bfy^{(1)}, \bfy^{(2)})$
with components $\bfy^{(j)}\in \R^{s_j}$, $j=1,2$.
Let us fix two vectors of  bases $\mathbf{b}^{(1)}=(b_1, \ldots, b_{s_1})$, and
$\mathbf{b}^{(2)}=(b_{s_1+1}, \ldots, b_{s_1+s_2})$,
with not necessarily distinct integers $b_i\ge 2$, $1\le i\le s$,
and let $\bfb=(\bfb^{(1)}, \bfb^{(2)})$.
Let $\mathbf{k} = (\mathbf{k}^{(1)}, \mathbf{k}^{(2)})$,
with components $\mathbf{k}^{(1)}\in \mathbb{N}_0^{s_1}$, and
$\mathbf{k}^{(2)}\in \mathbb{N}_0^{s_2}$.
The tensor product
$\xi_{\mathbf{k}}= w_{\mathbf{k}^{(1)}} \otimes \gamma_{\mathbf{k}^{(2)}}$,
where
$w_{\mathbf{k}^{(1)}}\in \mathcal{W}_{\mathbf{b}^{(1)}}$,
and $\gamma_{\mathbf{k}^{(2)}}\in \Gamma_{\mathbf{b}^{(2)}}$,
defines a function $\xi_{\mathbf{k}}$ on the $s$-dimensional unit cube,
\[
\xi_{\mathbf{k}}: [0,1)^s \rightarrow \mathbb{C},
\quad \xi_{\mathbf{k}}(\mathbf{x}) =
w_{\mathbf{k}^{(1)}}(\mathbf{x}^{(1)})  \gamma_{\mathbf{k}^{(2)}}(\mathbf{x}^{(2)})\;,
\]
where $\mathbf{x}= (\mathbf{x}^{(1)}, \mathbf{x}^{(2)})\in [0,1)^s$.

\begin{defn}
The family of functions
\[
\mathcal{W}_{\mathbf{b}^{(1)}}\otimes \Gamma_{\bfb^{(2)}} =
    \{ \xi_\bfk=w_{\bfk^{(1)}} \otimes \gamma_{\bfk^{(2)}},  \bfk= (\bfk^{(1)} ,
        \bfk^{(2)})\in \N_0^{s_1}\times \N_0^{s_2}\},
\]
is called a {\em hybrid function system} on $[0,1)^s$.
\end{defn}
It follows from \cite[Theorem 1 and Corollary 4]{Hel10c} and the techniques exhibited in
\cite{Hel11a} for non-prime bases $\bfb$
that such hybrid function systems are an orthonormal basis of $L^2([0,1)^s)$
and that a Weyl criterion holds.

\section{Results}\label{sec:results}

\begin{rem}
All of the following results remain valid if we change the order of the factors
in the hybrid function system,
as it will become apparent from the proofs below.
In particular, we may select some arbitrary $s_1$ coordinates and analyze them with the Walsh system
$\mathcal{W}_{\mathbf{b}^{(1)}}$ and treat the remaining $s_2$ coordinates with the $\bfb^{(2)}$-adic system.
\end{rem}

For an integrable function $f$ on $[0,1)^s$, the $\bfk$th Fourier
coefficient of $f$ with respect to the function system
$\mathcal{W}_{\mathbf{b}^{(1)}}\otimes \Gamma_{\bfb^{(2)}} $ is defined in the usual manner,
as the inner product of $f$ and $\xi_\bfk$ in $L^2([0,1)^s)$:
\[
\hat{f}(\bfk) = \int_{[0,1)^s}
f \overline{\xi_\bfk}\ d\lambda_s, \quad \bfk \in \N_0^s.
\]
The reader should notice that this definition encompasses the cases of Walsh and
$\mathbf{b}$-adic Fourier coefficients, by putting $s=s_1$ or $s=s_2$.

We denote the formal Fourier series of $f$ by $s_f$,
\[
s_f = \sum_{\bfk \in \N_0^s} \hat{f}(\bfk) \xi_\bfk,
\]
where, for the moment,
we ignore questions of convergence.

\begin{defn}
A {\em $\mathbf{b}$-adic elementary interval},
or {\em $\mathbf{b}$-adic elint} for short,
is a subinterval $I_{\mathbf{c},\mathbf{g}}$ of $[0,1)^s$ of the form
\[
I_{\mathbf{c}, \mathbf{g}}=\prod_{i=1}^s
    \left[ \varphi_{b_i}(c_i), \varphi_{b_i}(c_i) + b_i^{-g_i}  \right)\;,
\]
where the parameters are subject to the conditions
$\mathbf{g}=(g_1, \ldots, g_s)\in \mathbb{N}_0^s$,
$\mathbf{c}=(c_1, \ldots, c_s)\in \mathbb{N}_0^s$, and
$0\le c_i < b_i^{g_i}$, $1\le i\le s$.
We say that  $I_{\mathbf{c}, \mathbf{g}}$ belongs to the resolution class defined by $\mathbf{g}$
or that it has resolution $\mathbf{g}$.

A {\em $\mathbf{b}$-adic  interval} in the resolution class defined by $\mathbf{g}\in \N_0^s$
(or with resolution $\mathbf{g}$)
is a sub\-interval  of $[0,1)^s$ of the form
\[
	\prod_{i=1}^s \left[
    a_i b_i^{-g_i}, d_i b_i^{-g_i}
	\right),
    \quad 0\le a_i < d_i \le b_i^{g_i}, \;  a_i, d_i \in \mathbb{N}_0, \;1\le i\le s\; .
\]
\end{defn}

For a given resolution
$\mathbf{g}\in \mathbb{N}_0^s$,
we define the following domains:
\begin{align*}
\Delta_{\mathbf{b}}(\mathbf{g}) &= \left\{\mathbf{k}=(k_1, \ldots, k_s)\in \mathbb{N}_0^s :\
0\le k_i < b_i^{g_i}, 1\le i\le s  \right\},\\
\Delta_{\mathbf{b}}^*(\mathbf{g}) &= \Delta_{\mathbf{b}}(\mathbf{g})
    \setminus \{\mathbf{0} \}\;.\nonumber
\end{align*}
We note that $\Delta_\bfb(\bf0)=\{ \bfzero\}$.

\begin{rem}
\label{s:badic:rem:partition}
For a given resolution $\bfg \in \N_0^s$,
the family of $\mathbf{b}$-adic elints $\{I_{\bfc, \bfg}: \bfc\in \Delta_\bfb(\bfg)  \}$
is a partition of $\Is$.
\end{rem}

The following function will allow for a compact notation.
For $k\in \N_0$,
with $b$-adic representation $k=k_0 + k_1 b + \cdots$,
we define
\[
v_b(k) =
  \begin{cases}
   0 & \text{if } k=0, \\
   1 + \max\{j: k_j \neq 0\} & \text{if } k\ge 1.
  \end{cases}
\]
If $\bfk \in \N_0^s$, then let
$
v_\bfb(\bfk) = (v_{b_1}(k_1), \ldots, v_{b_s}(k_s)).
$

\begin{lem}\label{lem:piecewiseconstant}
For every $\bfk\in \N_0^s$,
$\xi_\bfk$ is a step function on $[0,1)^s$ and the following identity holds:
\begin{equation}
\label{eqn:mainID100}
\forall \bfx\in [0,1)^s: \quad \xi_\bfk(\bfx)
    =  \sum_{\bfc\in \Delta_\bfb(v_\bfb(\bfk))}
    \xi_\bfk(\varphi_\bfb(\bfc))
   \myone_{I_{\bfc,v_\bfb(\bfk)}}(\bfx).
\end{equation}
%valid for all $\bfx\in [0,1)^s$.
\end{lem}

\begin{proof}
For Walsh functions in base $\bfb^{(1)}$,
this follows from \cite[Remark (iii), p.211]{Hel93a}.
For the $\bfb^{(2)}$-adic functions,
we argue as in the proof of Lemma 3.5 in \cite{Hel10a}.
\end{proof}

\begin{cor}
For every $\bfk\neq \bfzero$,
the function $\xi_\bfk$ has integral $0$.
\end{cor}

\begin{proof}
If  $\bfk\neq \bfzero$,
then it is easy to prove that $ \sum_{\bfc\in \Delta_\bfb(v_\bfb(\bfk))} \xi_\bfk(\varphi_\bfb(\bfc)) =0$.
Taking the integral in (\ref{eqn:mainID100}),
the result follows.
\end{proof}

A key ingredient in the $\bfb$-adic approach to the theory of uniform distribution of sequences
is the study of the Fourier series of
indicator functions $\myone_I$ of $\bfb$-adic elints and intervals $I$.

\begin{lem}\label{lem:Fourier}
Let $I_{\bfc,\bfg}$ be an arbitrary $\mathbf{b}$-adic elint.
Then
\[
\hat{{\bf 1}}_{I_{\bfc,\bfg}}(\bfk) =
    \begin{cases}
   0 & \text{if } \bfk\not\in \Delta_\bfb(\bfg), \\
   \lambda_s(I_{\bfc,\bfg}) \overline{\xi_\bfk(\varphi_\bfb(\bfc))} & \text{if } \bfk\in \Delta_\bfb(\bfg).
  \end{cases}
\]
\end{lem}

\begin{proof}
The $s$-dimensional elint $I_{\bfc, \bfg}$ can be written as the cartesian product
$I_{\bfc^{(1)}, \bfg^{(1)}} \times I_{\bfc^{(2)}, \bfg^{(2)}}$.
As a consequence,
the Fourier coefficient can be written in the form
\[
\hat{\myone}_{I_{\bfc,\bfg}}(\bfk) = \hat{\myone}_{I_{\bfc^{(1)}, \bfg^{(1)}}}(\bfk^{(1)})\;
     \hat{\myone}_{I_{\bfc^{(2)}, \bfg^{(2)}}}(\bfk^{(2)}),
\]
where the first factor is the Fourier coefficient with respect to $\mathcal{W}_{\bfb^{(1)}}$ and
the second factor stems from $\Gamma_{\bfb^{(2)}}$.

It is an easy exercise to rewrite the proofs of Lemmas 1,2, and 3 in \cite{Hel93a} to handle
the Walsh system $\mathcal{W}_{\bfb^{(1)}}$.
This yields
\[
\hat{{\bf 1}}_{I_{\bfc^{(1)}, \bfg^{(1)}}}(\bfk^{(1)}) =
    \begin{cases}
   0 & \text{if } \bfk^{(1)}\not\in \Delta_{\bfb^{(1)}}(\bfg^{(1)}), \\
   \lambda_{s_1}(I_{\bfc^{(1)},\bfg^{(1)}})\; \overline{w_{\bfk^{(1)}}(\varphi_{\bfb^{(1)}}(\bfc^{(1)}))}
    & \text{if } \bfk^{(1)}\in \Delta_{\bfb^{(1)}}(\bfg^{(1)}).
  \end{cases}
\]
In the same manner,
we treat the $\bfb^{(2)}$-adic Fourier coefficients.
For this task,
is suffices to modify Lemmas 3.1, 3.3, and 3.5 in \cite{Hel09a} for the base $\bfb^{(2)}$.
We obtain
\[
\hat{{\bf 1}}_{I_{\bfc^{(2)}, \bfg^{(2)}}}(\bfk^{(2)}) =
    \begin{cases}
   0 & \text{if } \bfk^{(2)}\not\in \Delta_{\bfb^{(2)}}(\bfg^{(2)}), \\
   \lambda_{s_2}(I_{\bfc^{(2)},\bfg^{(2)}})\; \overline{\gamma_{\bfk^{(2)}}(\varphi_{\bfb^{(2)}}(\bfc^{(2)}))}
    & \text{if } \bfk^{(2)}\in \Delta_{\bfb^{(2)}}(\bfg^{(2)}).
  \end{cases}
\]
This proves the result.
\end{proof}

\begin{lem}\label{lem:pointwiseIdentity}
Let $I_{\bfc,\bfg}$ be an arbitrary $\mathbf{b}$-adic elint and
put $f=\myone_{I_{\bfc,\bfg}}$.
Then $f=s_f$ in the space $L^2([0,1)^s)$
and even pointwise equality holds:
\begin{equation}\label{s:badic:eqn:pointwiseIdentity}
\forall\bfx\in[0,1)^s: \quad
\myone_{I_{\bfc,\bfg}}(\bfx)
    = \sum_{\bfk \in \Delta_\bfb(\bfg)}
    \hat{\myone}_{I_{\bfc,\bfg}}(\bfk)\xi_\bfk(\bfx).
\end{equation}
\end{lem}

\begin{proof}
The pointwise identity between $\myone_{I_{\bfc^{(1)}, \bfg^{(1)}}}$ and its Fourier series with respect to
the Walsh system $\mathcal{W}_{\mathbf{b}^{(1)}}$ has been shown in the proof of Theorem 1 in \cite{Hel93a}.
For the function $\myone_{I_{\bfc^{(2)}, \bfg^{(2)}}}$ and the system $\Gamma_{\bfb^{(2)}}$,
the pointwise identity between the function and its Fourier series was proved in \cite[Lemma 2.11]{Hel11a}.

It is an easy exercise to rewrite the proof of Lemma 2.11 in \cite{Hel11a} for
the hybrid function system $\mathcal{W}_{\mathbf{b}^{(1)}}\otimes \Gamma_{\bfb^{(2)}} $.
This proves (\ref{s:badic:eqn:pointwiseIdentity}).
\end{proof}

\begin{rem}
Lemmas
\ref{lem:piecewiseconstant} - \ref{lem:pointwiseIdentity}
provide the necessary tools to
generalize Theorem A.11 in Dick and
Pillichshammer \cite[p. 562]{Dick10a} and,
in addition, exhibit a different method of proof.
For details,
we refer to Theorem 2.12 and Remark 2.13 in \cite{Hel11a}.
\end{rem}

\begin{lem}\label{lem:FCestimate}
Let $b\ge 2$, $0<\beta<1$ and put $J=[0,\beta)$.
Let $k\in \N$, $b^{g-1}\le k < b^g$, with $g\in \N$.
Then, in both cases,
for the Fourier coefficient relative to $\mathcal{W}_b$ as well as relative to
$\Gamma_b$,
we have the bound
\begin{equation}\label{eqn:FCestimate}
| \hat{\myone}_J(k)| \le \frac{1}{b^g \sin(\pi k_{g-1}/b)}.
\end{equation}
\end{lem}

\begin{proof}
Note first that the condition $b^{g-1}\le k < b^g$, $g\in \N$,
implies that $k=k_0+k_1 b+\dots+k_{g-1}b^{g-1}$ with $k_{g-1}\neq 0$.
For $\mathcal{W}_b$,
the bound (\ref{eqn:FCestimate}) follows from Lemma 2 in \cite{Hel93a}.
For $\Gamma_b$,
(\ref{eqn:FCestimate}) is shown by a verbatim translation of Lemma 3.3 and its proof in \cite{Hel09a}
from the case of a prime base to a general base.
\end{proof}

For an integer base $b\ge 2$,
we define the weight functions
\begin{align*}
\rho_b(k) \;&=\; \left\{
    \begin{array} {c@{\quad \mbox{  if  } \quad}l}
    1 & k=0, \\[0.5ex] \displaystyle{
        \frac{2}{b^{t}\sin(\pi k_{t-1}/b)}
        } & b^{t-1}\le k< b^t, \ t\in \N,
    \end{array} \right. \\
\text{and} & \nonumber\\
\rho_\bfb(\bfk)\;&=\; \sprod \rho_{b_i}(k_i), \quad \bfk=(k_1, \ldots,k_s)\in \N_0^s.
\end{align*}
We also introduce the weights $\rho^*$:
$\rho_b^*(0)=1$, $\rho_b^*(k)= \rho_b(k)/2$ for $k\ge 1$ and,
for $\bfk\in \N_0^s$,
$\rho_\bfb^*(\bfk)=\prod_{i=1}^s \rho_{b_i}^*(k_i)$.

\begin{cor}\label{cor:FCbound}
Let $I$ be an arbitrary $\bfb$-adic interval with resolution $\bfg\in \N_0^s$,
$I= \prod_{i=1}^s [a_i b^{-g_i}, d_i b^{-g_i})$,
with integers $a_i, d_i$, $0\le a_i<d_i\le b_i^{g_i}$, $1\le i\le s$,
and put $f=1_I -\lambda_s(I)$.

Then
\[
\forall\ \bfk \in \Delta_\bfb^*(\bfg): \quad
|\hat{f}(\bfk)| \le \rho_\bfb(\bfk),
\]
and, if $I$ is anchored at the origin,
i.e., if all $a_i$ are equal to 0,
then
\[
\forall\ \bfk \in \Delta_\bfb^*(\bfg): \quad
|\hat{f}(\bfk)| \le \rho_\bfb^*(\bfk).
\]
\end{cor}

\begin{proof}
For all $\bfx=(x_1, \ldots, x_s)\in [0,1)^s$ we have
\[
\myone_I(\bfx) = \prod_{i=1}^s \myone_{[a_i b^{-g_i}, d_i b^{-g_i}) } (x_i).
\]
Further,
$[a_i b^{-g_i}, d_i b^{-g_i})= [0, d_i b^{-g_i})\setminus [0, a_i b^{-g_i})$, $1\le i\le s$.
We then apply Lemma \ref{lem:FCestimate} to every coordinate.
This proves the result.
\end{proof}

If $\omega = ({\bf x}_n)_{n\ge 0}$ is
a -possibly finite- sequence in $[0,1)^s$ with at least $N$ elements,
and if $f:\, [0,1)^s$ $\to$ ${\C}$,
we define
\[
S_N(f,\omega) = \frac{1}{N} \sum_{n=0}^{N-1} f({\bf x}_n).
\]

Let $\mathcal{J}$ denote the class of all subintervals of $[0,1)^s$
of the form $\prod_{i=1}^s [u_i,v_i)$, $0\le u_i<v_i\le 1$, $1\le i\le s$,
and let $\mathcal J^*$ denote the subclass of $\mathcal{J}$ of intervals
of the type $\prod_{i=1}^s [0, v_i)$.
The extreme discrepancy and the star discrepancy of a sequence
are defined as follows (see Niederreiter \cite{Nie92a} for further information).

\begin{defn}
Let $\omega = (\mathbf{x}_n)_{n\ge 0}$ be a sequence in $[0,1)^s$.
\begin{enumerate}
\item The {\em (extreme) discrepancy}
$D_N(\omega)$ of the first $N$ elements of $\omega$ is defined as
\[
D_N(\omega) \;=\; \sup_{J\in \mathcal{J}} \left| S_N(\myone_J- \lambda_s(J), \omega)
    \right|.
\]

\item The {\em star discrepancy}
$D_N^*(\omega)$ of the first $N$ elements of $\omega$ is defined as
\[
D_N^*(\omega) \;=\; \sup_{J\in \mathcal{J}^*} \left| S_N(\myone_J - \lambda_s(J), \omega)
    \right|.
\]
\end{enumerate}
\end{defn}

\begin{thm}\label{thm:hybridETK}
Let $s=s_1+s_2$, with $s_1, s_2\in \N_0$, not both equal to 0.
Let $\bfb=(b_1, \ldots, b_s)$ be a vector of $s$ not necessarily distinct integers $b_i\ge 2$,
and let $\mathcal{W}_{\bfb^{(1)}}$ denote the Walsh system in base $\bfb^{(1)}$ and
$\Gamma_{\bfb^{(2)}}$  the $\bfb^{(2)}$-adic system in base $\bfb^{(2)}$,
$\bfb^{(1)}= (b_1, \ldots, b_{s_1})$, $\bfb^{(2)}= (b_{s_1+1}, \ldots, b_s)$.
%Put $\mathcal{F} =
Consider the hybrid function system
$\mathcal{W}_{\mathbf{b}^{(1)}}\otimes \Gamma_{\bfb^{(2)}} = \{ \xi_\bfk: \bfk\in \N_0^s\}$.

Then, for all $\bfg\in\N^s$,
\begin{equation}\label{eqn:ETK}
D_N(\omega) \;\le\; \epsilon_\bfb(\bfg) +
\sum_{\bfk \in \Delta_\bfb^*(\bfg)} \rho_\bfb(\bfk)
\left| S_N(\xi_\bfk, \omega) \right| \ ,
\end{equation}
and
\begin{equation*}\label{eqn:ETK*}
D_N^*(\omega) \;\le\; \epsilon^*_\bfb(\bfg) +
\sum_{\bfk \in \Delta_\bfb^*(\bfg)}  \rho_\bfb^*(\bfk)
\left| S_N(\xi_\bfk, \omega) \right| \ ,
\end{equation*}
where the error terms $\epsilon_\bfb(\bfg)$ and $\epsilon^*_\bfb(\bfg)$ are given by
\[
\epsilon_\bfb(\bfg) = 1 - \prod_{i=1}^s (1- 2 b_i^{-g_i}),
\quad
\epsilon^*_\bfb(\bfg) = 1 - \prod_{i=1}^s (1-  b_i^{-g_i}).
\]
\end{thm}

\begin{proof}
Our first step is a technical result.
Let $t_i, u_i\in [0,1]$, with the property $|t_i-u_i|\le \delta_i$,
where $\delta_i\in [0,1]$,  $1\le i\le s$.
One can show by the same method of proof as  in \cite[Lemma 3.9]{Nie92a}
that
\begin{equation}\label{eqn:inequality}
\left|  \prod_{i=1}^s t_i - \prod_{i=1}^s u_i \right| \le 1 - \prod_{i=1}^s (1- \delta_i).
\end{equation}

Next, consider the following approximation argument,
adapted from the proof of Theorem 3.6 in \cite{Hel09a}.
Suppose that $J$ is an arbitrary subinterval of $[0,1)^s$.
Let $\bfg=(g_1, \ldots, g_s)\in \N^s$ be arbitrarily chosen.
We consider the partition of $[0,1)^s$ given by the family of $\bfb$-adic
elints  $\mathcal{I}_\bfg =\{I_{\bfc, \bfg}: \bfc\in \Delta_\bfb(\bfg)  \}$.
Define
$\underline{J}$ as the union of those elints $I\in \mathcal{I}_\bfg$
that are contained in $J$,
$\underline{J} = \bigcup_{I: I\subseteq J} I$.
Further,
let $ \overline{J}$ denote the union of all elints $I\in \mathcal{I}_\bfg$
with nonempty intersection with $J$,
$ \overline{J} = \bigcup_{I:I \cap J \neq \emptyset} I$.
Then $\underline{J}\subseteq J\subseteq \overline{J}$,
where $\underline{J}$ may be void.
It is elementary to see that
\begin{align*}
\left| S_N(\myone_J - \lambda_s(J), \omega) \right| &\le
    \lambda_s(\overline{J}) - \lambda_s(\underline{J})\\
    &+ \max\left\{
        \left|S_N(\myone_{\underline{J}}-\lambda_s(\underline{J}), \omega)\right|,
        \left|S_N(\myone_{\overline{J}}-\lambda_s(\overline{J}), \omega)\right|
    \right\}.
\end{align*}

In every coordinate $i$,
the sidelength of $\underline{J}$ and $\overline{J}$
differs at most by $2b_i^{-g_i}$.
Hence, by an application of (\ref{eqn:inequality}),
we obtain the bound
\[
\lambda_s(\overline{J}) - \lambda_s(\underline{J}) \le \epsilon_\bfb(\bfg).
\]
The intervals $\underline{J}$ and $\overline{J}$ are both of the form
which was considered in Corollary \ref{cor:FCbound}.
There is only a finite number of such intervals.
Hence,
we obtain the bound
\begin{equation}\label{eqn:upperbound1}
\left| S_N(\myone_J - \lambda_s(J), \omega) \right|\le
 \epsilon_\bfb(\bfg) +
    \max_I \left| S_N(\myone_I - \lambda_s(I), \omega) \right|,
\end{equation}
where the maximum is taken over all intervals $I$ of the form
$I= \prod_{i=1}^s [a_i b^{-g_i}, d_i b^{-g_i})$,
with integers $a_i, d_i$, $0\le a_i<d_i\le b_i^{g_i}$, $1\le i\le s$.

The bound in (\ref{eqn:upperbound1}) is independent of the choice of $J$.
Further, each such $\bfb$-adic interval $I$ is a finite disjoint union of appropriate $\bfb$-adic elints
in $\mathcal{I}_\bfg$.
For this reason,
we may employ Lemma \ref{lem:pointwiseIdentity} and obtain
the following pointwise identity:
\begin{equation*}\label{s:badic:eqn:pointwiseIdentity2}
\forall\bfx\in[0,1)^s: \quad
\myone_{I}(\bfx)
    = \sum_{\bfk \in \Delta_\bfb(\bfg)}
    \hat{\myone}_{I}(\bfk)\xi_\bfk(\bfx).
\end{equation*}
The operator $S_N(\cdot, \omega)$ is linear in the first argument,
which yields
\[
S_N(\myone_I - \lambda_s(I), \omega) = \sum_{\bfk \in \Delta^*_\bfb(\bfg)}
    \hat{\myone}_{I}(\bfk) S_N(\xi_\bfk, \omega).
\]
Corollary \ref{cor:FCbound} implies
\begin{equation*}
D_N(\omega) \le \epsilon_\bfb(\bfg) +
    \sum_{\bfk\in \Delta^*_\bfb(\bfg)} \rho_\bfb(\bfk) \left| S_N(\xi_\bfk, \omega) \right|.
\end{equation*}

In the case of the star discrepancy $D_N^*(\omega)$,
the intervals $J$ are anchored at the origin.
This fact allows us %to use the bound on the Fourier coefficient of Lemma \ref{lem:FCestimate} and
to replace the weight function $\rho_\bfb$ by $\rho^*_\bfb$.
This finishes the proof.
\end{proof}

\begin{rem}
An elementary analytic argument shows that
$\epsilon_\bfb(\bfg)\leq 2s\delta$,
and $\epsilon^*_\bfb(\bfg)\leq s\delta$,
where $\delta=\max_{1\le i\le s} b_i^{-g_i}$.
\end{rem}

\begin{cor}\label{cor:ETK}
Let $\omega$, $\bfb$, and $\bfg$ be as in Theorem \ref{thm:hybridETK}.
Suppose that $B$ is a global bound for the exponential sums $S_N(\xi_\bfk, \omega)$
for all $\bfk$ in the finite domain $\Delta^*_\bfb(\bfg)$,
\[
\forall \bfk\in \Delta^*_\bfb(\bfg): \quad \left| S_N(\xi_\bfk, \omega) \right| \le B.
\]
Then
\begin{align*}
D_N(\omega )   \;&\le\; \epsilon_\bfb(\bfg) + B\cdot \prod_{i=1}^s (2.43 \;g_i\ln b_i +1),\\
D_N^*(\omega ) \;&\le\; \epsilon^*_\bfb(\bfg) + B\cdot \prod_{i=1}^s (1.22 \;g_i\ln b_i +1).
\end{align*}
\end{cor}

\begin{proof}
Let us first consider the extreme discrepancy $D_N(\omega)$.
The discrepancy bound
(\ref{eqn:ETK}) implies that we only have to estimate the sum of weights
\[
\sum_{\bfk\in\Delta^*_\bfb(g)} \rho_\bfb(\bfk) = \sum_{\bfk\in\Delta_\bfb(\bfg)} \rho_\bfb(\bfk) - 1.
\]
Because of the identities
\[
\sum_{\bfk\in\Delta_\bfb(\bfg)} \rho_\bfb(\bfk) =
    \prod_{i=1}^s  \sum_{k_i=0}^{b_i^{g_i}-1} \rho_{b_i} (k_i),
\]
and
\[
\sum_{k=0}^{b^g-1} \rho_b(k) = 1+ \sum_{t=1}^g \sum_{a=1}^{b-1}
    \sum_{k=ab^{t-1}}^{(a+1)b^{t-1}-1} \rho_b(k),
    \]
we obtain
\[
\sum_{\bfk\in\Delta_\bfb(\bfg)} \rho_\bfb(\bfk) =
    \prod_{i=1}^s (1 + 2g_i\cdot C(b_i)),
\]
where $C(b)= 1/b \sum_{a=1}^{b-1} 1/\sin(\pi a/b)$.
From Niederreiter \cite[p. 574, inequality (5)]{Nie76a} it follows that
$C(b) < (2/\pi) \ln b + 2/5$.
An elementary calculation gives the result.

The case of $D_N^*(\omega)$ is completely analogous.
\end{proof}

\begin{rem}
Corollary \ref{cor:ETK} generalizes  \cite[Corollary 4]{Hel93a} and \cite[Corollary 3.7]{Hel09a}.
\end{rem}

\section{Adding digit vectors}\label{s:addition}
% --------------------------
Any construction method for finite or infinite sequences of points
is based on arithmetical operations like addition or multiplication,
on a suitable domain.
It is most helpful if the algebraic structure underlying these operations is an abelian
group.
The choice of this group determines which function systems will be suitable
for the analysis of a given sequence,
because the construction method is intrinsically related to
function systems,
via the concept of the dual group (see Hewitt and Ross~\cite{Hewitt79a}).
Different types of
sequences require different types of function systems for their analysis.

It has been shown in \cite{Hel12c} that,
when we deal with digit vectors,
there exist {\em only two basic types} of addition of digit vectors:
addition without carry,
and addition with carry.
Further,
while there are only two types of addition
for a given length $m$ of the digit vectors,
there are at least $2^{m-1}$ different additions for such vectors.
They are obtained by mixing the two basic types.
This  number may be increased even further
if one employs  automorphisms of suitable groups of residues.
We refer the reader to \cite{Hel12c} for details.

Addition without carry is associated with Walsh functions and addition with carry
with the $b$-adic function system,
in the sense mentioned above, by considering  dual groups.
For this reason,
we have chosen a hybrid function system
composed of Walsh and $b$-adic functions for
the version of the \etk \  inequality that has been exhibited in Theorem \ref{thm:hybridETK},
in order to accomodate hybrid digital sequences.

\begin{rem}
It becomes apparent from the proofs above that our results can easily be extended to rather
`wild' additions of digit vectors, as they were considered in \cite{Hel12c}.
One only has to adapt the hybrid function system %$\mathcal{F}$
in a way that it corresponds
to the chosen additions.
\end{rem}

%\subsection*{Acknowledgements}
%The author would like to thank Harald Niederreiter, University of Salzburg, and RICAM,
%Austrian Academy of Sciences, Linz,
%for several helpful comments.

% ---------------------------------
%% Use the widest label as parameter.

%% In IMPAN journals, only the title is italicized; boldface is not used.
%% The issue number is only given when the issues are paginated separately.

%%%%%%%%%%% To ease editing, use normal size:

\normalsize
\baselineskip=17pt

%%%%%%%%%%%%%

% ---------------------------------
%\bibliographystyle{plain}
%\bibliography{../../literature/pLab,../../literature/crypt}
%\def\cdprime{$''$}
\def\cdprime{$''$} \def\cdprime{$''$} \def\cdprime{$''$}

\end{document}